\documentclass[reqno]{amsart}
\usepackage{epsfig,latexsym,amsfonts,amssymb,amsmath,amscd}
\usepackage{amsthm}
\usepackage{tikz-cd}
\usepackage[mathscr]{eucal}
\usepackage{mathrsfs}
\usepackage{mathbbol}
\usepackage{array}
\usetikzlibrary{matrix,arrows,decorations.pathmorphing}
\usepackage{oldgerm,units,color}

\newcommand{\jaiung}[1]{{\textcolor{red}{Jaiung: #1}}}
\newcommand{\nsets}{{\mathcal P}^*}
\usepackage{pst-node}
\usepackage{graphicx}

\usepackage{eepic, epic}

\usepackage{ifthen}

\def\mtw{\cdot_{\operatorname{tw}}}
\def\mtwt{{\operatorname{tw}^2}}

\def\bob{{\mathbf b}}

\def\bos{{\mathbf s}}

\def\tTz{ \tT_\operatorname{\zero}}

\newcommand{\lemref}[1]{Lemma~\ref{#1}}

\def\module0{module$^\dagger$}

\def\dag{\dagger}
\def\ssemiring0{$s$-semiring$^\dagger$}

\usepackage{tikz}
\usepackage{epsfig,latexsym,amsfonts,amssymb,amsmath,amscd,graphics,epic}
\usepackage{amsfonts,amssymb,amsmath,amscd,amsthm}
\usepackage[mathscr]{eucal}
\usepackage{mathrsfs}

\usepackage{mathbbol}

\usepackage{oldgerm,units}
\usepackage{wrapfig,epsfig}

\usepackage{ifthen}
\usepackage{mathbbol}

\usepackage{amsthm}
\usepackage[mathscr]{eucal}
\usepackage{mathrsfs}
\usepackage{mathbbol}
\usepackage{oldgerm,units}
\usepackage{wrapfig}

\newtheorem{theorem}{Theorem}[section]
\newtheorem{proposition}[theorem]{Proposition}
\newtheorem{definition}[theorem]{Definition}

\newtheorem{corollary}[theorem]{Corollary}

\newtheorem{example}[theorem]{Example}

\newcommand{\Net}{\mathbb N}

\newcommand{\one}{\mathbb{1}}
\newcommand{\zero}{\mathbb{0}}

\newcommand{\trop}[1]{\mathcal{#1}}

\newcommand{\tG}{\trop{G}}

\newcommand{\tT}{\trop{T}}

\newcommand{\ttT}{{_\tT}}

\newcommand{\Spec}{Spec}

    \ifx\proof\undefined
    \newenvironment{proof}{
    \smallskip
    \noindent\emph{Proof.}}{\hfill\(\Box\)
    \bigskip
    } \fi

\newcommand{\ifdef}[3]{\ifthenelse{\equal{#1}{true}}{#2}{#3}}

\definecolor{lgray}{gray}{0.90}

\def\Spec{\operatorname{Spec}}
\def\hSpec{\operatorname{Spec}_{\operatorname{Cong}}}
\def\Spece{\operatorname{Spec}_{e}}

\def\mcA{\mathcal A}

\def\hmcA{\widehat{\mathcal A}}

\def\mcF{\mathcal F}

\newtheorem{INote}[theorem]{Important Note}

\def\({\left(}
\def\){\right)}
\def\N{{\mathbb N}}
\def\Z{{\mathbb Z}}
\def\Q{{\mathbb Q}}

\def\pipe{{\underset{{\ \, }}{\mid}}}

\def\vsemifield0{$\nu$-semifield$^\dagger$}
\def\vsemiring0{$\nu$-semiring$^\dagger$}

\def\pipe1{{\underset{{1}}{\mid}}}
\def\lmod1{\mathrel  \pipe1  \joinrel \joinrel =}

\def\CFunFF1{\operatorname{CFun} (F,F)}

\def\semiring0{semiring$^{\dagger}$}
\def\Semiring0{Semiring$^{\dagger}$}
\def\Semirings0{Semirings$^{\dagger}$}
\def\semidomain0{semidomain$^{\dagger}$}
\def\semifield0{semifield$^{\dagger}$}
\def\semifields0{semifields$^{\dagger}$}
\def\vsemifields0{$\nu$-semifields$^{\dagger}$}
\def\domain0{domain$^{\dagger}$}
\def\predomain0{pre-domain$^{\dagger}$}
\def\predomains0{pre-domains$^{\dagger}$}
\def\domains0{domains$^{\dagger}$}

\def\vdomains0{$\nu$-domains$^{\dagger}$}

\def\domains0{domains$^\dagger$}

\def\mcH{\mathcal H}
\def\mcM{\mathcal M}

\def\MCong{\Cong_\mcM}

\newcommand{\etype}[1]{\renewcommand{\labelenumi}{(#1{enumi})}}
\def\eroman{\etype{\roman}}

\def\pipe{{\underset{{\tG}}{\mid}}}

\def\lmod{\mathrel  \pipe \joinrel \joinrel =}
\def\pipe{{\underset{{\tG}}{\mid}}}

\def\prima{bipotent-preserving}
\def\semiprime{geometric}

\def\Ann{{\operatorname{Ann}}\,}

\def\diag{\operatorname{diag}}

\newtheorem{thm}[theorem]{Theorem}

\newtheorem*{thm*}{Theorem}
\newtheorem{lem}[theorem]{Lemma}
\newtheorem{rem}[theorem]{Remark}
\newtheorem{prop*}{Proposition}
\newtheorem{conj*}{Conjecture}

\newtheorem{prop}[theorem]{Proposition}
\newtheorem{defn}[theorem]{Definition}
\newtheorem*{examp*}{Example}
\newtheorem*{examples*}{Examples}
\newtheorem*{remark*}{Remark}
\newtheorem*{defn*}{Definition}
\newtheorem*{note*}{Note}

\def\R{\mathbb R}
\def\C{\mathbb C}

\def\la{\lambda}

\def\tT{\mathcal T}

\def\tTz{\tT_\zero}

\numberwithin{equation}{section}

\def\M0{M_{\zero}}

\def\tGz{\mathcal G_\zero}

\def\ttT{ {\tT}^\natural}

\def\PS{P}

\def\Cong{\Phi}

\def\Diag{{\operatorname{Diag}}}

\def\semirings0{semirings$^\dagger$}

\newcommand{\nPS}[1]{\PS_{(!#1)}}
\newcommand{\nPSo}[1]{\nPS{\one}}

\newcommand{\absl}[1]{|{#1}|}
\newcommand{\adj}[1]{\operatorname{adj}({#1})}

\def\textbfb{\textbf{b}}
\def\textbfc{\textbf{c}}
\def\textbfd{\textbf{d}}

\def\textbfx{\textbf{x}}
\def\textbfy{\textbf{y}}
\def\textbfz{\textbf{z}}

\begin{document}


\title[Spectrum of prime congruences]
{The spectrum of prime congruences  of a semiring pair}

%
%
\author[L.~Rowen]{Louis Rowen}
\address{Department of Mathematics, Bar-Ilan University, Ramat-Gan 52900, Israel} \email{rowen@math.biu.ac.il}

\makeatletter
\@namedef{subjclassname@2020}{%
    \textup{2020} Mathematics Subject Classification}
\makeatother


\date{\today}



\thanks{The author was supported by the ISF grant 1994/20 and the Anshel Peffer Chair.}
\thanks{The author thanks Jaiung Jun for helpful comments on an earlier version, and also thanks the referee for detailed comments on restructuring the paper.}


%



\begin{abstract} Continuing the study of the structure of semirings, we turn to the spectrum of prime congruences.  Joo and Mincheva developed an elegant theory in the special case of idempotent semirings, which is generalized here to  ``semiring pairs,'' which include supertropical semirings and various classes of hyperrings. Our main result is that the Joo-Mincheva spectrum can be embedded into the prime spectrum of a semiring pair, and the mapping is onto when the pair satisfies a mild restriction which we call ``positive $e$-type.''
\end{abstract}

\maketitle





\section{Introduction}
This is part of an ongoing project  to find a general  algebraic framework for
semiring theory. Our initial motivation came from tropical algebra, in which the max-plus algebra played a critical role at the outset, and followed by $F1$-geometry \cite{Lo}. These algebras $\mcA$ are all
\textbf{idempotent semirings} in the sense that $b+b=b$ for all $b\in \mcA$.
But the structure theory of semirings is quite challenging, largely because of the lack of negation, and such basic properties such as unique factorization of polynomials, multiplicativity of determinants, and the characteristic polynomial of a matrix, all fail. (In fact in the max-plus algebra, the sum of two nonzero elements is {\it never} zero!)

In order to recover these properties, Izhakian and Rowen \cite{IR} developed supertropical algebra, in which one builds an algebra from two isomorphic copies, one \textbf{tangible} and one \textbf{ghost}. But supertropical algebra is but one instance of   \textbf{hyperfields}, introduced by Marty~\cite{Mar}, used effectively by Krasner~\cite{krasner}, with many nice recent applications including \cite{BL}, and   generalized more recently in \cite{NakR}.

In the instances mentioned above, except for $F_1$-algebra, the underlying semiring is not idempotent, so a broader theory is in order.

The examples above were unified in 2016 in \cite{Row21}, in which it was noted that the element $\zero$ could be replaced by a $\tT$-submodule $\mcA_0$ of $\mcA,$ and negation by
a \textbf{negation map} $(-)$ which satisfies $a (-)a\in \mcA_0$ instead of $a(-)a = \zero.$
Thus, $\mcA_0$ takes the place of $\zero$, and the element $e:= \one (-)\one$ is of special significance. (or, multiplicatively, $\one$) in classical mathematics, and $\zero$ has no significant role except as a place marker in linear algebra. (Note that any bipotent algebra lacks a negation map except in the meaningless case  $\mcA_0=\mcA.$)

For example, a root of a polynomial $f$ now would be an element $a$ such that $f(a)\in \mcA_0.$ The ensuing structure $(\mcA,\mcA_0)$ is called \textbf{pair}; a pair with a negation map is called a \textbf{triple}, and when also endowed with a ``surpassing relation'' replacing equality, is called a \textbf{system}. Pairs satisfying ``Property N,'' to be described below, seem to satisfy the minimalist set of axioms needed to obtain a workable theory. The linear algebra of pairs is studied in \cite{AGR1}.

The connection of systems to hyperfields is particularly noteworthy; both structures permit two types of morphisms, cf.~\cite{V} and \cite{AGR2} which shows how one can pass back and forth from power sets of hyperfields to certain systems.

This basic notion also is applicable to other algebraic structures such as Lie semialgebras, cf.~\cite{GaR}.

In this paper we turn to the prime spectrum. In the theory of universal algebra, ideals are replaced by congruences, in order to describe homomorphic images, so one needs a spectrum of prime congruences  generalizing the prime spectrum of   a ring.

Joo and Mincheva   \cite{JM,FM} have developed a beautiful theory of the spectrum in the special case of commutative, idempotent semirings, describing chains of   congruences which are  prime with respect to the ``twist product.''

It turns out that their results apply to a considerably wider class of semirings, which we call pairs of \textbf{$e$-type},
cf.~Definition~\ref{cha0}, which means $\one + ke  = ke$ for some $k>0.$ This class obviously includes idempotent domains, $F_1$-algebra and supertropical algebras, as well as power sets of many kinds of hyperrings (described in \cite{AGR2}) such as the phase hyperfield,  the weak phase hyperfield, and the signed and unsigned tropical hyperfields, Viro's complex hyperfield, and obviously all finite hyperfields.  But this class also is closed under polynomial extensions (as opposed to the class of hyperrings).

In Theorem~\ref{sp2} we obtain a lattice injection from the Joo-Mincheva  spectrum of prime congruences of a suitable idempotent image of $\mcA$ into the prime spectrum of $\mcA$,  which  by Theorem~\ref{rd} is an  isomorphism in the case of pairs of positive      $e$-type.

On the other hand,
in \S\ref{prop1} we also  consider  other significant congruences which do not arise in \cite{JM}, namely those not containing $(\one,e)$.

Although our emphasis is on semirings, unfortunately the power set of a hyperfield  need not satisfy distributivity of multiplication over addition, so we also need to cope with non-distributive structures with associative addition and multiplication, which we call \textbf{nd-semirings}.

\section{Underlying algebraic structures}$ $

\begin{definition} $\Net$ denotes the  natural numbers.
\eroman
\begin{enumerate}


\item 
An \textbf{additive semigroup} is an abelian semigroup, with the operation  ``$+$" and the neutral element~$\zero.$

\item
A additive semigroup $\mcA$ is \textbf{idempotent} if $b+b=b$ for all $b\in \mcA.$ $\mcA$ is \textbf{bipotent} if $b_1+b_2 \in \{b_1,b_2\}$ for all $b_i\in \mcA.$

\end{enumerate}
 \end{definition}


\begin{example}
     For an ordered additive semigroup  $(\tG,+)$, the \textbf{max-plus algebra} on $\tGz :=\tG \cup \{-\infty\}$ is given by defining multiplication to be the old addition on $\tG,$ with $-\infty$, the new zero element, multiplicatively  absorbing.
   The max-plus algebra is~idempotent.

\end{example}

    The example used most frequently is for $\tG = (\mathbb R,+)$ with the usual order (denoted $\mathbb T$ in \cite{JM}); other common instances are for $\tG = \Q$, $\tG = \mathbb Z,$ and $\tG = \N.$ 

  \subsection{$\tT$-modules}$ $

We follow \cite{AGR1}; also see \cite{JMR}.

  \begin{definition}\label{Tmagm}
Let $\tT$ be a   monoid\footnote{In \cite{AGR1}, $\tT$ is just a set, in order to permit Example~\ref{wa1}(v) below;  in that case    we adjoin a formal absorbing element~$\zero_\tT$, so that $ \tTz = \tT \cup \{\zero\}$ is a monoid.} with a designated element $\one_\tT$.
\eroman
\begin{enumerate}

 \item  A  $\tT$-\textbf{module} $\mathcal A$ is an additive semigroup
$(\mathcal A,+,\zero)$ with a  $\tT$-action $\tT\times \mathcal A \to \mathcal A$, for which   \begin{enumerate}

    \item $\one_\tT b = b,$ for all $b\in \mcA$.

 \item    $(a_1a_2)b = a_1(a_2b)$ for all $a_i\in \tT,$ $ b\in \mcA$.

  \item $a(b_1+b_2) = ab_1 +ab_2,\quad \text{for all}\; a \in \tT,\; b_i \in \mathcal A.$

 \item $\zero$ is  absorbing, i.e. $a \zero  = \zero, \  \text{for all}\; a \in \tT.$
\end{enumerate}

The elements of $\tT$ are called \textbf{tangible}.

\item   From now on in this paper, for convenience, in contrast to \cite{JMR}, we assume that $\tT \subseteq \mcA$, since we will be focusing on semirings. We define $\tTz =\tT\cup \{\zero\},$ and declare $\zero \mcA = \zero.$ This makes  $\mcA$ a $\tTz$-module.

\item In a   $\tT$-module $\mcA$, define $\one = \mathbf{1} := \one_\tT,$ and inductively $\mathbf{k+1} = \mathbf{k} +\one_\tT,$ $\forall k\in \N.$


     \item
 An \textbf{nd-semiring} is an additive semigroup with another  operation, multiplication, denoted as concatenation, with  $\zero$ multiplicatively absorbing. 


 \item
 A \textbf{semiring} \cite{golan92} is an nd-semiring that satisfies all the properties of a ring (including associativity and distributivity of multiplication over addition), but without
negation.

\item
A semiring $\mcA$ is a \textbf{semifield} if $(\mcA\setminus \{\zero\},\cdot)$ is a group.

 \item An  \textbf{nd-$\tT$-semiring}\footnote{``nd'' is short for ``not distributive,'' even though the action of $\tT$ does distribute over addition in $\mcA.$
 } (resp.~$\tT$-\textbf{semiring}) is an nd-semiring (resp.~semiring) $\mcA$ which is a   $\tT$-module with $\tT$ central in $\mcA$, i.e., every element of $\tT$ commutes and associates with all the elements of $\mcA,$ and $\one=\one_\tT$ is the multiplicative unit.

\item A   $\tT$-module $\mcA$ is called  \textbf{admissible} if $\mcA$ is spanned by $\tTz$.

 \item  $\ttT$ is defined to be the   subset of $(\mathcal A,+,\zero)$ additively spanned by~$\tTz.$
 In other words,  $\tT \subseteq \ttT,$ and when $b_1,b_2\in \ttT$,
 then $b_1+b_2\in \ttT.$ The \textbf{height} $h(t)$ of an element $b\in \ttT$ is defined inductively:
$h(\zero)=0,$ $h(a)=1$ for all $a\in \tT$, and $$h(b) = \min\{(h(b_1)+h(b_2): b =b_1+b_2, \ b_i \in \ttT\}.$$
\end{enumerate}
\end{definition}

\begin{rem}  We can replace $\mcA$ by $\ttT,$ and thereby  reduce often to the case that
    $\mcA$ is admissible.
\end{rem}

 \subsection{Pairs}$ $

 \begin{definition}\label{symsyst}
We follow \cite{AGR1,JMR}, suppressing $\tT$ in the notation when it is understood. \eroman
 \begin{enumerate}
   \item
An  \textbf{nd-semiring  pair}  $(\mcA,\mcA_0)$   is  an  nd-$\tT$-semiring $\mcA$ given together with an $\tT$-submodule $\mcA_0$ containing~$\zero$.

\begin{INote}
    In contrast to  \cite{AGR1, JMR, Row24b}, from now on in this paper, ``pair'' exclusively means `` nd-semiring pair.''
\end{INote}

  \item A~\textbf{semiring pair} $(\mcA,\mcA_0) $ is   a  pair for which  $\mcA$ is a semiring.
\item
$(\mcA,\mcA_0)$ is of the \textbf{first kind} if $a + a \in \mcA_0$ for all $a\in \tT$,
 and of the \textbf{second kind} if  $a + a \notin \mcA_0$ for some $a\in \tT$.


\item
 $(\mcA,\mcA_0)$ is \textbf{cancellative} if it satisfies the following two conditions for $a\in \tT,$ $b,b_1,b_2\in \mcA$:
\begin{enumerate}
  \item
If $a b \in \mcA_0$, then $b\in \mcA_0.$
\item
If $a b _1= ab_2$, then  $b_1=b_2.$
\end{enumerate}



 \item A \textbf{gp-pair} is a pair for which  $\tT$ is a  group. (Thus every gp-pair is cancellative.)

\item
A pair $(\mcA,\mcA_0)$ is  \textbf{proper} if $\mcA_0\cap \tTz =\{ \zero\}$.

  \item A  pair  $(\mathcal
A, \mcA_0)$ is \textbf{shallow} if $\mcA  = \tT \cup \mcA_0.$

\end{enumerate}
\end{definition}

\begin{lem}  Suppose   $(\mcA,\mcA_0)$ is a   pair.
\begin{enumerate}\eroman
    \item $(\mcA,\mcA_0)$  is of the first kind if $\mathbf{2} = \one + \one \in \mcA_0$.
\item For $\mcA$ cancellative, $(\mcA,\mcA_0)$  is of the second kind if and only if $\mathbf{2} \notin \mcA_0$.
\end{enumerate}\end{lem}
\begin{proof}
    (i) $a+a = a(\one+\one) \in \mcA_0
$ for $a\in \tT$.

    (ii) If $a(\one+\one) = a+a \notin \mcA_0$ for $a\in \tT$, then $\one+\one  \notin \mcA_0$.
\end{proof}

 Fractions are described rather generally in \cite[\S 4.1]{JMR}, which shows that when $\mcA$ is cancellative,  $\tT^{-1}\tT$~is the group of fractions of $\tT$, and  replacing $\tT$ by $\tT^{-1}\tT$ and $(\mcA,\mcA_0)$ by the $\tT^{-1}\tT$-pair $(\tT^{-1}\mcA,\tT^{-1}\mcA_0)$,
    we may reduce to the case of gp-pairs.

\begin{INote}\label{Note1} Philosophically, $\mcA_0$ takes the place of $\zero$ (or, multiplicatively, $\one$) in classical mathematics. The significance is that since nd-semirings need not have negation (for example, $\Net$), $\zero$ has no significant role except as a place marker in linear algebra. \end{INote}

 \begin{example}\label{wa1} (Many examples are   given below in \S\ref{mp}, but here is a quick preliminary taste.)
\begin{enumerate}\eroman
         \item  (\textbf{The  classical pair})    $\mcA_0 = \{\zero\}$. (More generally,   $\mcA_0$ is   an ideal of a ring $\mcA$. )

  \item  (\textbf{The  degenerate pair})    $\mcA_0 = \mcA .$

\item  For any nd-$\tT$-semiring $\mcA$, define $\mcA_0 = \{b+b: b\in \mcA\}.$ Note that the pair $(\mcA,\mcA_0)$  is improper if $\mcA$ is idempotent.  In previous work \cite{AGR2, AGR1,JMR} we assumed that all pairs are proper, to distinguish from the degenerate case. But this is precisely the case treated so successfully in~\cite{JM}, so we include it here.

\item   Lie pairs are studied in \cite{GaR}.

\item $\mcA$ is the matrix algebra $M_n(F),$ and  $\tT  $ is the set of  matrix units, with the identity matrix adjoined. The product of two matrix units could be $\zero,$  but $ \tTz = \tT \cup \{\zero\}$ is a monoid.
\end{enumerate}
 \end{example}

\subsubsection{Property N}\label{propN1} $ $

\begin{definition}{\cite[\S3.1]{AGR1}}\label{propN0}$ $
We say that a  pair
$(\mcA,\mcA_0)$ satisfies
 \textbf{Property~N}  if there is an element $\one^\dag\in \tT$  such that
 \begin{enumerate}\eroman
     \item $e:=\one+\one^\dag\in \mcA_0$.
     \item If $\one + a \in \mcA_0,$ then $\one + a = e.$
   \item  For each $b\in \mcA$, defining
 $b^\dag = \one^\dag b$,  and putting $b^\circ = b+ b^\dag$, then $b^\circ \in \mcA_0$.
 \end{enumerate}
\end{definition}

\begin{rem} \label{ker3}
   The  max-plus pair   $(\mathcal
A, \zero) := (\tGz,\{-\infty\})$ is a  proper, shallow, bipotent pair over $\tT=\tG$, but lacks Property~N.
   \end{rem}

\begin{INote}
    From now on, we assume that any given pair  $(\mcA,\mcA_0)$  also
satisfies   Property~N. Note that  $e$ must be uniquely determined, although $\one^\dag$ need not  be uniquely determined. An example is given in Remark~\ref{ker4}. An idempotent example:
 Let $\mcA = \tTz \cup {\infty}$ where $\zero$ is an additive identity and  $a+a = a$, but $a+a' =\infty$ for each $a\ne a'\in \tT$; then $\mcA_0 = \{\zero, \infty\}.$
\end{INote}

 \begin{lem}[{\cite[Lemma~3.10]{AGR1}}]\label{est}  $e\one^\dag = e.$
\end{lem}

   \begin{definition}$ $
 \begin{enumerate}\eroman
\item
    The \textbf{distributive center $Z(\mcA)$ of $\mcA$}, is the set of elements that commute, associate, and distribute over all elements of $\mcA.$

\item $(\mcA,\mcA_0)$  is $e$-\textbf{distributive} if
$b +k b^\circ = (1+ke)b$  for all $b\in \mcA$ and all $k\in \N$.

\item $(\mcA,\mcA_0)$  is
 $e$-\textbf{central} if  $(\mcA,\mcA_0)$  is
 $e$-{distributive} and $e(b_1b_2) = (eb_1)b_2 = b_1(eb_2)=(b_1b_2)e$ for all $b\in \mcA,$ i.e., $e\in Z(\mcA).$

\item  $(\mcA,\mcA_0)$  is $e$-\textbf{idempotent} if $e=e+e$.
\end{enumerate}
   \end{definition}

So any multiplicatively associative  $e$-{distributive} pair is $e$-{central}.
 \begin{lem}[{\cite[Lemma~3.5]{AGR1}}]\label{esq}
     If $(\mcA,\mcA_0)$  is $e$-distributive then
 $e(b_1+b_2) =e b_1 +e b_2 $ for all  $b_i\in  \mcA_0$. In particular,  $e^2 = e+e$.  \end{lem}

Thus, $\mcA_0$ is idempotent when  $(\mcA,\mcA_0)$  is $e$-distributive and $e$-idempotent.

\begin{rem}
    If $z\in Z(\mcA)$, then $\sum _i (b_i z)b_i' = (\sum _i b_i b_i')z$ for all $b_i,b_i'\in \mcA,$ and $\mcA z \triangleleft \mcA.$
\end{rem}

\subsubsection{The characteristic and the $e$-type}$ $

 \begin{definition}{\cite[Definition~2.5 (ii),(iii)]{AGR1}}\label{cha0} $ $   \begin{enumerate}\eroman
     \item A  $\tT$-module  $\mcA$ has       \textbf{characteristic} $(p,k)$ if $\mathbf{p+k} =\mathbf{k},$ for  the smallest possible $p>0$ (if it exists), and then the smallest such $k\ge 0.$

      \item    $(\mcA,\mcA_0)$ has      $\mcA_0$-\textbf{characteristic} $k>0$ if ${k}\mcA\subseteq \mcA_0,$ for  the smallest such $k.$  A pair  $(\mcA,\mcA_0)$ has      $\mcA_0$-\textbf{characteristic} $0$ if there is no such $k>0.$

        \item  A pair  $(\mcA,\mcA_0)$ has        $e$-\textbf{type} $ (k,k')$ for $0<k'\le k$ if
        $b+k b^\circ =k' b^\circ,$ $\forall b\in \mcA$. $(\mcA,\mcA_0)$ has        $e$-\textbf{type} $ k$ if $k'=k.$

       \item    A pair
   $(\mathcal
A, \mcA_0)$  is
 \textbf{$e$-final} if $(\mathcal
A, \mcA_0)$  has $e$-type $\one$, i.e.,
 $b +b^\circ=b^\circ,$ $\forall b\in \mcA$.
 \end{enumerate}
 \end{definition}
 \begin{rem}$ $
   \begin{enumerate}\eroman
       \item   Any idempotent pair has characteristic $(1,1)$ but has $\mcA_0$-characteristic $0$  since $ a+a=a.$

\item   Any pair of the first kind  has $\mcA_0$-characteristic $0$.

 \item A pair need not have        $e$-type.  If $\mcA$ has characteristic
   $(1,k)$   then $(\mcA,\mcA_0)$ has        $e$-type at most $k$, since $b +kb^\circ = b+\mathbf kb + \mathbf kb^\dag  = \mathbf kb + \mathbf kb^\dag = k b^\circ.$
   \end{enumerate}
 \end{rem}

 \begin{lem}
     Any $e$-final pair is $e$-idempotent.
 \end{lem}
 \begin{proof}      $e+e =  \one^\dag + (\one +e) =   \one^\dag + e = \one^\dag +  (\one^\dag)^\circ  = (\one^\dag)^\circ   = e.$
 \end{proof}



\subsubsection{Negation maps}$ $

 A \textbf{negation map} $(-)$ on $(\mcA,\mcA_0)$ is
 a module automorphism of order $\le 2$ , such that
$(-)\tT=\tT$, $(-)\mcA_0 = \mcA_0,$ $b+((-)b)\in \mcA_0$ for all $b\in \mcA,$   and $$(-)(ab) = ((-)a)b =
a((-)b),\qquad  \forall a\in \tT, \quad b\in \mcA.$$

We write $b_1(-)b_2$ for $b_1+((-)b_2),$ and $b_1 (\pm) b_2$  for $\{ b_1 +b_2, b_1(-)b_2\}.$

\begin{lem}
When $(\mcA,\mcA_0)$ is a semiring pair,
the negation map satisfies
\begin{equation}\label{eq: neg}(-)(bb') = ((-)b)b' = b((-)b'), \quad \forall b,b' \in \mcA.\end{equation}
\end{lem}
\begin{proof}
    Immediate from distributivity.
\end{proof}

\begin{INote} The identity map is a negation map on  $(\mcA,\mcA_0)$  when $(\mcA,\mcA_0)$ is of the first kind, as in Example~\ref{wa1}(ii), thereby enabling us
to lift the theory of pairs with negation map to arbitrary pairs. However, one often is given a negation map;
      we  always take $\one^\dag = (-)\one$.
\end{INote}

\subsubsection{Homomorphisms}$ $

 We consider an nd-$\tT$-semiring $\mcA$ and a $\tT'$-nd-semiring $\mcA'$.
\begin{definition}$ $ \begin{enumerate}\eroman
    \item
A map $f:\mcA \to \mcA'$  is \textbf{multiplicative} if $f(a)\in \tT'$ and $f(ab) = f(a) f(b)$ for all $a\in \tT, \ b \in \mcA.$
 \item     A  \textbf{homomorphism} $f:\mcA \to
    \mcA'$ is a multiplicative  map satisfying $f(b_1+b_2) =  f(b_1)+f(b_2)$ and $f(b_1b_2) = f(b_1) f(b_2)$, $\forall b_i\in \mcA.$
A \textbf{projection} is an onto homomorphism.
\end{enumerate}
\end{definition}

Here is an important instance for the sequel.

\begin{lem}
    \label{emul}
      Suppose  $(\mcA,\mcA_0)$ is an $e$-central and $e$-idempotent  pair. Then $\mcA_0= \mcA e $ is an idempotent nd-semiring with unit element $e$, and the map $a \mapsto ae$ defines a projection
      $\mcA \to \mcA e $.
\end{lem}
\begin{proof}
$e^2 =e$, so, by $e$-distributivity,
    $\mcA_0= \mcA e$ is an nd-semiring with unit  element $e$.
   $ (b_1+b_2)e = b_1e+b_2e$ and    $ (b_1e)(b_2e) =     b_1b_2e^2= b_1b_2e.$
\end{proof}



\subsection{Doubling}$ $

Define  $\hat{\mcA} := \mcA\times \mcA $, which will play two key roles, one in providing a negation map and the other in analyzing congruences.
 We refer to \cite[\S 5.2]{AGR2}, adapting an inspired idea of \cite{JM}.
  We always write a typical element of $\hmcA$ as $\mathbf{b}= (b_1,b_2)$.

\begin{definition}\label{sw}$ $
\begin{enumerate}\eroman
\item
Given an  nd-$\tT$-semiring $\mcA$,   let $\widehat{\tT } = ({\tT }\times 0 )\cup (0
\times {\tT }) \subset \hat{\mcA} := \mcA\times \mcA .$

Define the \textbf{twist
product} on $\hmcA$ by
\begin{equation}\label{twi}
\mathbf{b}\mtw \mathbf{b'}:= (b_1b'_1 + b_2b'_2, b_1b'_2+b_2b'_1), \quad \textrm{for } \mathbf{b} = (b_1,b_2),\ \mathbf{b'} = (b_1',b'_2)\in\mcA.
\end{equation}

\item
We write $\mathbf b^\mtwt$ for $\mathbf b\mtw \mathbf b$.
\item
The \textbf{switch map} on $\hmcA$ is defined by $(-)\mathbf{b}= (-)(b_1,b_2)  = (b_2,b_1)$.
\item
$\Diag $ denotes the ``diagonal''   $\{(b,b): b\in \mcA\}.$

\end{enumerate}
\end{definition}

\begin{lem}
    $\hat{\tT}$ is a monoid with respect to the twist product and having unit element $(\one, \zero)$, generated by $(\zero,\one)\tT$.
\end{lem}
\begin{proof} $(a_1, \zero)\mtw (a_2, \zero) = (a_1 a_2, \zero)$ and $(a_1, \zero)\mtw (\zero, a_2) = (\zero,a_1 a_2)$, etc. Hence $\hat{\tT}$ is a monoid with
  $\hat{\tT} = (\one,\zero)\tT \,\cup \,(\zero,\one)\tT$, and   $(\zero,\one)^\mtwt = (\one, \zero).$ The rest is clear.
\end{proof}

\begin{lem}\label{twass}   $\hat{\mcA}  $ is a $\widehat{\tT } $-nd-semiring. When $\mcA$ is a semiring,
the twist product on $\hmcA$ is associative, and makes $\hmcA$ a semiring. [See \cite{JM} and \cite[Lemma~3.1]{JMR} for a proof.]
\end{lem}

\begin{rem}
    \label{gen}
    If $\textbfz = (z,z)\in \Diag,$ then $\textbfb\mtw \textbfz = (b_1z + b_2z, b_1z + b_2z) = (b_1+b_2)\textbfz.$
\end{rem}

In other words, $\Diag$ is an ideal of $\hmcA$, using the twist product.
Here is a way of providing a negation map.

\begin{lem} There is a natural embedding of $\mcA$ into $\hmcA$ given by $b \mapsto (b,\zero).$ $(\hmcA, \Diag )$ is a pair, endowed with a negation map, namely the switch map.  In this case, $ (-)(\one,\zero) = (\zero,\one),$ and $e = (\one,\one).$
\end{lem}
\begin{proof} The first assertion is easy. Thus $\one$ is identified with $(\one,\zero),$ and $(-)(\one,\zero) = (\zero,\one)$
 so   $e = (\one,\one).$
\end{proof}

In this way, one may think of $\mathbf{b}$ as $b_1(-)b_2$, where the second component could be interpreted as the negation of the first.

\section{Main kinds of pairs}\label{mp}

This section reviews different kinds of pairs, to indicate their diversity.

\subsection{Metatangible and $\mcA_0$-bipotent pairs}$ $

\begin{definition} We generalize ``idempotent'' and ``bipotent'' respectively.
\begin{enumerate}   \eroman

\item A \textbf{metatangible pair}  is a  pair $(\mcA,\mcA_0)$
in which
     $a_1+a_2 \in \tT \cup \mcA_0$ for any $a_1,a_2 $
in~$\tT.$

\item
A  metatangible pair $(\mcA,\mcA_0)$ is $\mcA_0$-\textbf{bipotent}   if $a_1 + a_2 \in \{ a_1 , a_2 \} \cup \mcA_0$
    for all $a_1 , a_2  \in \tT$.
\end{enumerate}
\end{definition}

\begin{lem}
    \label{etype} If $(\mcA,\mcA_0)$ is  $e$-distributive and shallow, and $\one + ke \in \mcA_0$, then $(\mcA,\mcA_0)$ has  either      $e$-type $k$ or  $e$-type $(k,1)$.
\end{lem}
\begin{proof} Take $k$ minimal such; we may assume that $k>0.$
    Then $a:=\one + (k-1)e \in \tT.$ Let  $b= a+\one^\dag.$ Then $ b +\one = a+e = \one + ke \in \mcA_0.$

    If $b\in \mcA_0$ then $b= (\one^\dag)^\circ = e, $ so $ e+ \one  = a+e = \one + ke , $ so we may replace $k$ by $\one.$
    Now $\one + e = (\one +\one)(-)\one,$
    which is either $\one^\circ =e,$ or else $\one +\one \in \mcA_0$ so is $e,$
    implying again $\one + e = e.$ Thus, in this case, $(\mcA,\mcA_0)$ has       $e$-type 1.

    Otherwise $b\in \tT,$ and $\one + ke = a+e =  b +\one =\one^\circ = e,$ and $(\mcA,\mcA_0)$ has       $e$-type $(k,1).$
\end{proof}


\subsection{Examples of $\mcA_0$-bipotent pairs}$ $

Much of the theory of pairs concerns $\mcA_0$-bipotent pairs,  which are appropriate to tropical geometry, so we provide a range of such examples in this subsection, starting with a familiar one treated in \cite{JM}.

 \subsubsection{Supertropical pairs} $ $

  \begin{definition} [{\cite[Example 5.9]{AGR2}}] \label{supert}   Suppose $\tG$ is an ordered multiplicative monoid with absorbing minimal element~$\zero_{\tG}$, and $\tT$ is a monoid, together with an  monoid homomorphism $\nu: \tT\to \tG$.
Take the action $\tT\times\tG\to \tG$ defined by $a\cdot g = \nu(a)g.$
Then $\mcA := \tT\cup  \tG$ is a multiplicative monoid  whose multiplication subsumes $\nu$ and the given multiplications on $\tT$ and on $\tG.$

We  define addition on $\mcA$ by $$b_1+b_2 = \begin{cases}
 b_1 \text{ if } \nu(b_1)>\nu(b_2),\\
b_2 \text{ if } \nu(b_1)<\nu(b_2),\\
\nu(b_1) \text{ if } \nu(b_1)=\nu(b_2).
\end{cases}.$$

\begin{rem} \label{ker4}   $\nu$ extends to a projection   $\mcA \to \tG$.
Let $\mcA_0 =  \tG,$ and $e = \one_\tG.$  The pair    $(\mathcal
A, \mcA_0)$ is    proper of the first kind, shallow, and  $e$-final $\mcA_0$-bipotent. $(\mathcal
A, \mcA_0)$ has characteristic $(1,2)$ and $\mcA_0$-characteristic~$2$ (since $e = \one_\tT +\one_\tT =e')$.  We call $(\mcA,\mcA_0)$ the \textbf{supertropical pair arising from} $\nu.$
\end{rem}

\begin{enumerate}\eroman
    \item  The simplest nontrivial case is for $\tG = \ \{ \zero, e\}$:
 \begin{enumerate}\eroman
    \item
 For $\tT = \{\one\},$
we modify the  semifield  $\tTz= \{ \zero, \one\} $  to
the \textbf{super-Boolean pair}, defined as  $(\mcA,\mcA_0)$ where $\mcA= \{ \zero,\one,e\}$ with
$e$  additively absorbing, $\one+\one  =e,$ and $\mcA_0 = \{ \zero, e\}.$

 \item At the other extreme, \cite[Example~2.21]{AGR2} is the  pair $(\mcA ,\mcA_0 )$ with $\nu :  \tT \to   \{  e\}$ the constant map. Thus
  $a_1+a_2= \{e\}$ for all $a_1,a_2 \in \tT.$ This pair   is {\it not} metatangible.
\end{enumerate}
 \item The \textbf{standard supertropical pair} has $\tG = \tT \cup \{\zero\}$ and $\nu = id_\tT.$

 \item Other  nontrivial maps such as $\absl{\phantom{w}}: \C^\times \to \R^\times$   give variants  which we do not explore here.

\end{enumerate}
  \end{definition}

 \subsubsection {The truncated pair}\label{trun}

 \begin{definition} \label{trc}
     Fixing $m>0,$ take an ordered group $\tG$ and the supertropical pair $(\tT\cup  \tG, \tG)$, and form the \textbf{$m$-truncated}   pair  $({\tT}_m , \tG _m) $  where ${\tT}_m = \{ a\in \tT : a\le m\},$ $\ \tG_m = \nu({\tT}_m)\cup \zero_\tG$, with the supertropical addition  and multiplication except that we put $a_1 a_2=m$, $a_1\nu(a_2) = \nu(a_1) a_2  = \nu(a_1)\nu(a_2) = \nu(m)$ when the    product of $\nu(a_1)$ and $\nu(a_2)$ in $\tG$ is greater than $m$.
 \end{definition}

\begin{rem}
 $(\mathcal
A, \mcA_\zero) := ({\tT}_m\cup  \tG_m  , \tG _m)$ is a  proper, shallow,   $e$-final $\mcA_0$-bipotent pair. \end{rem}

\subsubsection{The minimal  $\mcA_0$-bipotent pair of a monoid}$ $

\begin{example}\label{minbp}
    $\tT$ is an arbitrary   monoid, $\mcA = \tTz \cup \{\infty\},$ $\mcA_0 =  \{\zero,\infty\},$ and $b_1+b_2= \infty$ for all $b_1 \ne b_2$ in $\tT \cup \{\infty\}$.
We call these the \textbf{minimal  $\mcA_0$-bipotent pairs}. There are two kinds:
\begin{itemize}
  \item First kind, where $a+a = \infty$ for all $a\in \tT$.
  \item Second  kind, where $a+a = a$ for all $a\in \tT$.
\end{itemize}

\end{example}

 \subsection{Hypersemirings and hyperpairs}\label{hype}$ $

\begin{defn}[\cite{Mar}]\label{Hyp00} $\mcH$ is a multiplicative monoid with absorbing element $\zero,$ and $\tT$ is a submonoid of~$\mcH.$ $\mathcal{P}(\mcH)$ denotes the power set of  $\mcH,$
    and $\nsets = \mathcal{P}(\mcH) \setminus \emptyset.$
\begin{enumerate}\eroman
    \item

$\mathcal{H}$ is a $\tT$-\textbf{hypersemiring} if $\mcH$ also is  endowed with a binary operation $\boxplus :\mcH \times \mcH \to \nsets $,\footnote{\cite{NakR} permits $\boxplus :\mcH \times \mcH \to \mathcal{P}(\mcH).$} satisfying the properties:

\begin{enumerate}\eroman
   \item The operation
$\boxplus$ is  associative  and abelian in the sense that  $(a_1
\boxplus a_2) \boxplus a_3 = a_1 \boxplus (a_2\boxplus a_3)$  and $a_1
\boxplus a_2= a_2
\boxplus a_1 $ for all
$a_i$ in $\mathcal{H}.$
  \item
 $\zero \boxplus a = a,$ for every $a\in \mathcal{H}.$
 \item  $\boxplus $ is extended to $\nsets $, by defining, for     $S_1,S_2\in  \nsets ,$   $S_1 \boxplus  S _2:= \cup _{s_i \in S_i}\,  s_1 \boxplus  s_2.$
   \item  The   natural  action of  $\mcH$  on  $ \nsets $ ($aS = \{as: s\in S\}$) makes $ \nsets $ an $\mcH$-module.
\end{enumerate}
We call $\boxplus$ ``hyperaddition.''  We view $\mcH \subseteq \nsets $ by identifying $a$ with $\{ a\}$.

 \item
 A \textbf{hypernegative} of an element $a$ in a hypersemiring $(\mathcal{H},\boxplus,\zero)$ (if it exists) is an
element  ``$-a$'' for which $\zero \in a \boxplus (-a)$.
 If the hypernegative $-\one$ exists in $\mcH,$ then we define $e = \one \boxplus (-\one).$

  \item
A \textbf{hyperring}\footnote{In \cite{NakR} this is called ``canonical''.} is a   hypersemiring $\mathcal{H}$ for which every element~$a$
has a unique {hypernegative}  $-a$, whereby, for all $a_i\in \mcH,$

\begin{enumerate}
   \item $(-)(a_1\boxplus a_2) = (-)a_2 \boxplus (-)a_1.$
   \item $-(-a_1) =a_1.$
 \end{enumerate}

  \item
  A  hyperring $\mathcal H$ is a \textbf{hyperfield} if $\mathcal H \setminus \{\zero\}$ is a  multiplicative group.
  \end{enumerate}
 \end{defn}
 \begin{rem}[\cite{AGR2,Row21}] \label{hp1} $ $
 \begin{enumerate}\eroman

    \item Take any $\tT$-sub-hypersemiring $S_0$ of $\mcA:=\nsets $ for which $S_0 \cap \mcH= \{\zero\}.$ Then we get a proper pair $(\nsets ,\nsets _0)$, where
   $\nsets _0 = \{S \subseteq \mcH : S_0 \cap S \ne \emptyset\}.$

  \item Suppose $S_0 = \{\zero\}.$  Then
$\nsets _0 = \{S \subseteq \nsets  : \zero \in S \} $.

 \item    If each $a\in \mcH$ has a unique hypernegative, then $(\nsets ,\nsets _0)$ has a negation map given by applying the hypernegative element-wise.

 \item  $\mcH$ need not span $\nsets$. On the other hand, the span of~$\mcH$ need not be closed under multiplication.  We define the \textbf{hyperpair of} $\mcH$ to be  the sub-pair of $(\nsets ,\nsets _0)$ generated by $\mcH$.

\item  $\nsets$ also is~multiplicatively associative, elementwise.

  \item
  $S(\boxplus S_i)\subseteq \boxplus (SS_i) $ for $S,S_i\in \nsets$, by \cite[Proposition 1]{MasM}, although $\nsets$ need not be distributive.
\end{enumerate}\end{rem}

\begin{rem}\label{hyprop} Notation as above, suppose that $\mcH$ has unit element $\one,$ and let $e =\one \boxplus (-)\one.$
\begin{enumerate} \eroman

\item
When $\mcH$ is a group and $e =\mcH \setminus \{\one\}, $ as in \cite[Proposition 2]{MasM} with $|\mcH|>2$, then $a(-)a = \mcH -a$ for any invertible $a\in \mcH$, implying $a^\circ a^\circ = \mcH$,
so   $\mcA$ has  $e$-type $2$.
\item
\cite[Proposition 3]{MasM} gives instances of  hyperfields with $e = \{\zero, \one, -\one\} $, in which case
 $(\nsets ,\nsets _0)$ is $e$-idempotent and   $e$-final.
\end{enumerate}
\end{rem}

Celebrated examples  of    hyperfields (and their accompanying hyperpairs) are reviewed in \cite[Examples~3.11 and 3.12]{Row24b}, including the  tropical hyperfield,  hyperfield of signs, and phase hyperfield.

\subsubsection{Residue hyperrings and hyperpairs}$ $

The following definition was inspired by Krasner~\cite{krasner}.
\begin{definition}\label{kras1}
    Suppose  $\mcA$ is    a  $\tT$-semiring and $G$ is a   multiplicative subgroup of $\tT \subseteq \mcA$, which is \textbf{normal} in the sense that $b_1G b_2G = (b_1b_2)G$ for all $b_i\in \mcA$. Define the \textbf{residue hypersemiring} $\mcH =\mcA/G$  over $\tT/G$
    to have multiplication induced by the cosets, and \textbf{hyperaddition} $\boxplus : \mcH \times \mcH \to \mathcal{P}(\mcA) $ by $$b_1 G \boxplus b_2 G = \{ cG: c\in b_1G + b_2G\}.$$
\end{definition}

When $\mcA$ is a field, the residue hypersemiring is called  the ``quotient hyperfield'' in the literature.

\begin{example}
    \label{Examplesofquotient} Here is part of the huge assortment of examples of quotient hyperfields   given in \cite[\S 2]{MasM}.
We  take $\mcH = \mcA/G,$ and its hyperpair $(\mcA,\mcA_0),$ as in Remark~\ref{hp1}(iv).  \begin{enumerate}\eroman
\item
$G = \{ \pm 1\}$. Then $\zero \in \one \boxplus \one ,$ so $(\mcA,\mcA_0)$ has $\mcA_0$-characteristic 2 and is multiplicatively idempotent.

\item  The \textbf{tropical hyperfield}
  is identified with the quotient hyperfield $\mcF/G$,
  where $\mcF$ denotes a field with a surjective non-archimedean
  valuation $v: \mcF \to \R\cup\{+\infty\}$, and
  $G:=\{f\in \mcF: v(f) =0\}$, the equivalence
  class of any element $f$ whose valuation $a$ is
  identified with the element $-a$ of~$\mcH$.  The {tropical hyperfield} is $e$-final.

  \item  The \textbf{Krasner hyperfield} is  $F/F^\times$.

\item  The \textbf{hyperfield of signs} is $\R/\R^+$, and is $e$-final.

\item  The \textbf{phase hyperfield} can be identified  with the quotient hyperfield
$\C/\R_{>0}$, and is $e$-final.

\end{enumerate}
\end{example}

 $\mcA/G$ need not be $e$-distributive; the phase hyperfield is a counterexample.

\subsection{The function pair}$ $

Here is a construction which significantly enhances the class of pairs under consideration.
Given a $\tT$-module~$\mcA$ and a monoid $S,$ define $\mcA^S$ to be the set of functions   $f:S \to \mcA$ of finite support, i.e., $f(s) = \zero$ for almost all $s\in S$.
If $(\mcA,\mcA_0)$ is a pair over $\tT,$ then $(\mcA^S,\mcA_0^S)$ is a pair over the elements of $\tT^S$ (the nonzero functions with image in $\tTz$ having support 1), and is of the same $e$-type as $(\mcA,\mcA_0)$, seen by checking elementwise. When $\mcA$ is a semiring, we can define \textbf{convolution multiplication} on~$\mcA^S,$ given by $(f * g)(s) = \sum _{s's''=s} f(s')g(s'').$
This construction applies to polynomials ($S=\N$), Laurent series ($S= \Z$), matrices ($S$ is a set of matrix units, although here we have to adjoin the identity matrix to $\tT^S$), and so~forth. When $(\mcA,\mcA_0)$ is $e$-central, $(\mcA^S,\mcA_0^S)$ also  is $e$-central.

\section{Congruences}\label{Co1}$ $

Classically, one defines homomorphic images of an algebraic structure  $\mcA$  by defining a
\textbf{congruence} of~$\mcA$  to be an
equivalence relation $\Cong$ which, viewed as a set of ordered
pairs, is a subalgebra of $\hmcA $ (under componentwise multiplication).
 In our case, we also require that $\Cong$  be a $\tT$-submodule of $\hmcA$.

  \begin{rem}\label{congp}
         For any  pair
 $(\mcA,\mcA_0)$, any congruence $\Cong$ of $\mcA$ can be applied
to produce a  pair  $(\overline{\mcA},\overline{\mcA_0})= (\mcA,\mcA_0)/\Cong,$ where
$\overline{\mcA} =\mcA/\Cong$ and, writing $\bar b$ for the image of $b\in \mcA,$ we define $a\bar b  = \overline{ab},$    and
$\overline{\mcA_0}= \{ \bar b  \in \overline{\mcA} : (b,\mcA_0)\cap \Cong\ne \emptyset\}.$
 \end{rem}

For example,
$\Diag $ can be viewed as the \textbf{trivial} congruence. Here is a trick that leads to a generalization of \cite{JM}.

\begin{definition} For a pair $(\mcA,\mcA_0)$ satisfying Property~N, \begin{enumerate}\eroman
    \item
       a $(\one,e)$-\textbf{congruence}  on a pair $(\mcA,\mcA_0)$ is a congruence  containing $(\one,e)$.

      \item
$\Diag_e $ denotes the intersection of all $(\one,e)$-congruences.
\end{enumerate}
\end{definition}

 \begin{lem}\label{id1}  If
    $\Cong$ is a $(\one,e)$-congruence on a  pair $(\mcA,\mcA_0)$ satisfying Property~N, then \begin{enumerate}\eroman
      \item   $\bar e=\bar \one$. Hence the pair  $(\mcA,\mcA_0)/\Cong$ is degenerate.

     \item   $\overline{ \one^\dag}  = \bar \one,$ and $\overline{\mcA}$ is idempotent.
 \end{enumerate}
 \end{lem}
\begin{proof}
 (i)    Obviously $\bar e=\bar \one$, so $b= be \in \mcA_0$ for all $b\in \mcA$.

 (ii)
    $\bar \one = \bar e = \overline{\one^\dag}\bar e = \overline{\one^\dag},$ so $\bar  e +\bar  e = \bar \one +\overline{\one^\dag}=\bar  e,$ and thus $\bar \one = \bar \one+\bar \one.$
\end{proof}

\begin{lem}\label{tr1}  $ $\begin{enumerate}\eroman
  \item  There is a lattice injection from the   congruences   of the idempotent pair $(\mcA, \mcA_0)/\Diag_e$ to the   congruences of     $(\mcA, \mcA_0)$.

    \item The canonical map $\Phi \mapsto \Phi e$ is a lattice homomorphism from the   congruences $\Cong$ of  an $e$-central pair $(\mcA,\mcA_0)$ to  a congruence $\Psi$ of  $\mcA e$,
      which restricts to a lattice isomorphism from the   $(\one,e)$-congruences of $(\mcA,\mcA_0)$ when $(\mcA,\mcA_0)$ is $e$-final.
\end{enumerate}
\end{lem}
\begin{proof} (i) By Lemma~\ref{id1}.

 (ii)   Given  a congruence $\Cong$, clearly $\Cong e = \{ (b_1e,b_2e) :(b_1,b_2)\in \Cong\}  $ is a congruence of $\mcA e$.

    Conversely, given  a congruence  $\Psi$ of $\mcA e$, define $\Cong = \{ (b_1,b_2) \in \hmcA :(b_1e,b_2e) \in \Psi\} $, noting that $(b_1,b_2)\in \Cong$ if and only if $(b_1e,b_2)\in \Cong$, if and only if $(b_1e,b_2e)\in \Cong.$
\end{proof}

To proceed further, one wants a congruence which contains a given element $\textbfb$ of $\hmcA$.
 If  $\textbfb = (b_1,b_2)$, then $(\mcA \textbfb,\mcA \textbfb)    +\Diag$ is a subalgebra of $\hmcA$ which satisfies reflexivity and symmetry, but not necessarily transitivity. At times this failure is remedied by the following intriguing construction, inspired by  ~\cite[Lemma~3.7]{JM}. Remark~\ref{gen} plays a key role, using the twist product $\mtw$.

\begin{definition}
    \label{prin} Suppose $\mcA$ is an nd-semiring, with given $\textbfb= (b_1,b_2)\in \hmcA$ satisfying $(c(b_1+b_2))c'= c((b_1+b_2)c') $ for all $c,c'\in \mcA.$
  Define $$\Cong_{\textbfb} = \{(x,y): \,(x,y)+\sum (\textbfc_i (b_1+b_2,b_1+b_2))\textbfc_i'\in \Diag , \text{ for some } \textbfc_i,\textbfc_i'\in  \hmcA\}.$$
Noting that $ (b_1+b_2,b_1+b_2) = \textbfb (-) \textbfb$ (where we are using the switch map of Definition~\ref{sw}), we see when $b_1+b_2 \in Z(\mcA)$ (in particular when $\mcA$ is a commutative semiring) that $$\Cong_{\textbfb} = \{(x,y):\, (x,y)+  \textbfc  (\textbfb (-)(\textbfb)) \in \Diag , \text{ for some } \textbfc \in  \hmcA\}.$$
\end{definition}

\begin{lem}
    In Definition~\ref{prin}, when $(\mcA,\mcA_0)$ is an $e$-central pair, one could replace $\textbfc_i$ by $\textbfc_i e$.
\end{lem}
\begin{proof}
    If $(x,y)+\sum \textbfc_i (b_1+b_2,b_1+b_2)\textbfc_i'\in \Diag $ then $$ (x,y) +\sum \textbfc_i e(b_1+b_2,b_1+b_2)\textbfc_i'  = (x,y)+\sum \textbfc_i (b_1+b_2,b_1+b_2)\textbfc_i'(-)\sum \textbfc_i (b_1+b_2,b_1+b_2)\textbfc_i'\in \Diag , $$ since $\sum \textbfc_i (b_1+b_2,b_1+b_2)\textbfc_i''\in \Diag$ by Lemma~\ref{gen}.
\end{proof}

\begin{lem}
      (generalizing  \cite[Lemma~3.7]{JM}) When $(\mcA,\mcA_0)$ has    $e$-type  $(k,k'), $ $\Cong_{\textbfb} $ is a congruence  which contains $\textbfb$  under either of the conditions:
      \begin{enumerate}\eroman
          \item $\mcA$ is a semiring.
           \item $b_1+b_2 \in Z(\mcA)$.
      \end{enumerate}
\end{lem}
\begin{proof} We show (i);  (ii) is easier since the calculations collapse.
    $\Cong_{\textbfb} $ is obviously reflexive (taking $\textbfc_i = \textbfc_i' =0$) and symmetric.

    To show that $\Cong_{\textbfb} $ is transitive, write $\textbfx = (x_1,x_2)$ and $\textbfy = (x_2,x_3)$,  $\textbfc_i = (c_{i,1},c_{i,2})$, and  $\textbfc_i' = (c'_{i,1},c'_{i,2})$. If $\textbfx +\sum \textbfc_i (b_1+b_2,b_1+b_2)\textbfc_i' \in \Diag $, we see  by Lemma~\ref{gen} that
    $$x_1+ \sum  (c_{i,1}+c_{i,2} ) (b_1+b_2)(c_{i,1}'+c_{i,2} )= x_2+\sum (c_{i,1}+c_{i,2} ) (b_1+b_2)(c_{i,1}'+c_{i,2}').$$  Likewise if $(x_2,x_3)+\sum \textbfd_i (b_1+b_2,b_1+b_2)\textbfd_i'\in \Diag $ then

    $x_2 + \sum  (d_{i,1}+d_{i,2} ) (b_1+b_2)(d_{i,1}'+d_{i,2} ) \sum  = x_3 + \sum  (d_{i,1}+d_{i,2} ) (b_1+b_2)(d_{i,1}'+d_{i,2} ) ,$ so together
     \begin{equation}
         \begin{aligned}
x_1+ \sum  (c_{i,1}+c_{i,2} )& (b_1+b_2)(c_{i,1}'+c_{i,2} ) + \sum  (d_{i,1}+d_{i,2} ) (b_1+b_2)(d_{i,1}'+d_{i,2} )\\& = x_2+\sum (c_{i,1}+c_{i,2} ) (b_1+b_2)(c_{i,1}'+c_{i,2}') + \sum  (d_{i,1}+d_{i,2} ) (b_1+b_2)(d_{i,1}'+d_{i,2} )\\ &=x_3+\sum (c_{i,1}+c_{i,2} ) (b_1+b_2)(c_{i,1}'+c_{i,2}') + \sum  (d_{i,1}+d_{i,2} ) (b_1+b_2)(d_{i,1}'+d_{i,2} ),
         \end{aligned}
     \end{equation}    implying $(x_1,x_3)\in \Cong_{\textbfb}$. The only other nontrivial verification is that if $(x_1,x_2)\in \Cong_{\textbfb}$ and $(y_1,y_2)\in \Cong_{\textbfb}$ then writing    $$x_1+ \sum  (c_{i,1}+c_{i,2} ) (b_1+b_2)(c_{i,1}'+c_{i,2} )= x_2+\sum (c_{i,1}+c_{i,2} ) (b_1+b_2)(c_{i,1}'+c_{i,2}')$$    and  $$y_1+ \sum  (d_{i,1}+d_{i,2} ) (b_1+b_2)(d_{i,1}'+d_{i,2} )= y_2+\sum (d_{i,1}+d_{i,2} ) (b_1+b_2)(d_{i,1}'+d_{i,2}'),$$
     we get
\begin{equation}
    \begin{aligned}
        x_1y_1 &+  \sum  (c_{i,1}+c_{i,2} ) (b_1+b_2)(c_{i,1}'+c_{i,2}')y_1 + \sum  x_1(d_{i,1}+d_{i,2} ) (b_1+b_2)(d_{i,1}'+d_{i,2}') \\& \qquad \qquad   + \sum  (c_{i,1}+c_{i,2} ) (b_1+b_2)\left((c_{i,1}'+c_{i,2}')\sum (d_{i,1}+d_{i,2} ) (b_1+b_2)(d_{i,1}'+d_{i,2}')\right)\\&=x_2y_2+  \sum  (c_{i,1}+c_{i,2} ) (b_1+b_2)(c_{i,1}'+c_{i,2}')y_1 + \sum  x_1(d_{i,1}+d_{i,2} ) (b_1+b_2)(d_{i,1}'+d_{i,2}') \\& \qquad \qquad   + \sum  (c_{i,1}+c_{i,2} ) (b_1+b_2)\left((c_{i,1}'+c_{i,2}')\sum (d_{i,1}+d_{i,2} ) (b_1+b_2)(d_{i,1}'+d_{i,2}')\right),
    \end{aligned}
\end{equation}
implying $(x_1y_1,x_2y_2)\in \Phi.$
  If $(\mcA,\mcA_0)$ has $e$-type $(k,k'),$ then $$\textbfb + k (b_1^\circ+b_2^\circ,b_1^\circ+b_2^\circ) = (b_1 + k(b_1^\circ+b_2^\circ), b_2+k(b_1^\circ+b_2^\circ))  =  (k'(b_1^\circ+b_2^\circ),k'(b_1^\circ+b_2^\circ))  \in \Diag.$$
\end{proof}

\subsection{ Strongly prime, prime, radical, and semiprime congruences}$ $

This section is a direct generalization of
 \cite[\S 2]{JM}, without assuming additive idempotence.
The twist product, utilized in \cite{JM} in similar situations,  is also a key tool here.

\begin{definition}\label{twist2}$ $ Suppose that $\mcA$ is an nd-semiring.
\begin{enumerate}\eroman
\item
The \textbf{twist product} $\Cong_1 \mtw \Cong_2 := \{ \bob \mtw \bob ' : \bob\in \Cong_1 ,\, \bob' \in \Cong_2\}.$
We write $\Cong^\mtwt  $ for $ \Cong\mtw \Cong.$

\item
A congruence $\Cong$ of $\mcA$ is \textbf{semiprime} if it satisfies the following   condition: For a congruence $\Cong_1 \supseteq  \Cong$, if $\Cong_1 ^\mtwt \subseteq \Cong$  then $\Cong_1  = \Cong$.

\item
A congruence $\Cong$ of $\mcA$ is \textbf{radical} if it satisfies the following   condition: $\textbfb {^\mtwt} \in \Cong$ implies $\textbfb\in \Cong.$
\item
A congruence $\Cong$ of $\mcA$ is \textbf{prime} if it satisfies the following  condition: for  congruences  $\Cong_1$, $\Cong_2 \supseteq  \Cong$ if $\Cong_1 \mtw \Cong_2 \subseteq \Cong$, then $\Cong_1
= \Cong$ or $\Cong_2 = \Cong$.
\item
A congruence $\Cong$ of $\mcA$ is \textbf{strongly prime} if it satisfies the following   condition: If $\textbfb\mtw \textbfb' \in \Cong$ then $\textbfb \in \Cong$ or $\textbfb' \in \Cong$. (This is called ``prime'' in \cite{JM}, but ``prime'' and ``strongly prime'' are the same for commutative semiring pairs of positive $e$-type.)

\item
A congruence $\Cong$ of $(\mcA,\mcA_0)$ is $\tT$-\textbf{cancellative} if
whenever $a\textbfb \in \Cong$ for $a\in \tT$ then $\textbfb \in \Cong.$

\item
A congruence $\Cong$ of $(\mcA,\mcA_0)$ is \textbf{irreducible} if
whenever $\Cong_1 \cap \Cong_2 = \Cong$ for  congruences $\Cong_1$,
$\Cong_2 \supseteq  \Cong$ then $\Cong_1  = \Cong$ or $\Cong_2 =
\Cong$.

\item We say that a pair $(\mcA ,\mcA_0)$ is \textbf{reduced}, resp.~\textbf{semiprime}, resp.~\textbf{a domain}, resp.~\textbf{prime}, resp.~\textbf{irreducible}, if the trivial congruence is radical, resp.~semiprime,  resp.~{strongly prime}, resp.~{prime},  resp.~irreducible.
\end{enumerate}
\end{definition}

\begin{definition}        Take $S_1 = \Cong$ and inductively $\{S_{i+1}=\textbfb\in \mcA:\textbfb ^{\mtwt}  \in S_i\}.$  Then define $\sqrt{\Cong}:= \cup_{i\ge 1} S_i$.
\end{definition}

\begin{rem}
    \label{semip0} The following observations are direct consequences of the definition:
\begin{enumerate} \eroman
     \item Any  strongly prime congruence is prime.

     \item A congruence $\Cong$ is radical iff $\Cong =\sqrt{\Cong}$.
\end{enumerate}
\end{rem}

(This is too naive for a decisive radical theory of congruences; \cite[Example 3.3]{JM} shows that $\sqrt \Cong$ need not be   a congruence, even for commutative semirings, since transitivity may fail.)

\begin{lem}\label{bf}  (Analogous arguments as in \cite[Propositions 2.2, 2.6]{JM})
With the same notation as above, one has the following.
\begin{enumerate}\eroman
\item
    For a congruence  $\Cong$, if $\textbfb\in \hmcA$ and  $\mathbf{b'}\in \Cong$, then $\textbfb\mtw \mathbf{b'}\in \Cong.$

\item A congruence $\Cong$ is prime if and only if it is semiprime and irreducible.
  \item
 The intersection of semiprime congruences is a semiprime congruence.
   \item
 The intersection of radical congruences is a radical congruence.
\item The union or intersection of a chain of congruences is a congruence.
 \item The union or intersection of a chain of (resp.~ strongly prime, prime, radical, semiprime)
 congruences is a  (resp.~ strongly prime, prime, radical, semiprime) congruence.

  \end{enumerate}
\end{lem}








Here is  an obvious general method for constructing prime congruences, which also works in the noncommutative setting.

\begin{lem}
    \label{commpr}
[{Generalizing~\cite[Theorem~3.9]{JM}, \cite[Proposition~3.17]{JuR1}}] Suppose   $S\subseteq \hmcA$
is any subset disjoint from  $\Diag(\mcA),$ satisfying  the property that for any congruences $\Phi_1,   \Phi_2,$ if $S\cap\Phi_1, S \cap \Phi_2$ are nonempty then $S \cap (\Phi_1\mtw \Phi_{2})\ne \emptyset.$  Then  there is
a congruence
$\Cong$ of $\mcA$ maximal with respect to being disjoint from~$\hat
S$, and $\Cong$ is prime.
\end{lem}
 \begin{proof}
By Zorn's lemma there is a  congruence
$\Cong$ of $\mcA$ maximal with respect to being disjoint from   $\hat
S$. We claim that $\Cong$ is prime. Indeed, if $\Cong_1 \mtw \Cong_2
\subseteq \Cong$ for  congruences $\Cong_1 , \Cong_2 \supset
\Phi,$ then   $S \cap(\Cong_1 \mtw \Cong_2)\ne \emptyset,$ a contradiction.
 \end{proof}


There is a lattice injection from the spectrum of prime congruences of a suitable idempotent semiring (\cite{JM}) into $\hSpec(\mcA)$,  which we shall see in Theorem~\ref{rd} is an  isomorphism in the important case of pairs having positive     $e$-type. We also
   investigate  congruences which do not arise in \cite{JM}, but which still preserve $\mcA_0$-bipotence.

 \begin{definition}\label{pcs}$ $
The \textbf{spectrum of prime congruences} $\hSpec (\mcA)$ is the set  of  prime congruences of  $\mcA$. \end{definition}

Izhakian~\cite{I} also has investigated congruences in supertropical algebra, with a somewhat different version of prime spectrum.

\subsection{The spectrum of prime congruences for pairs of positive $e$-type}\label{crad}$ $

The spectrum of prime congruences is our main
 subject of interest in this paper, generalizing    the prime spectrum of   a ring.
$\hSpec (\mcA)$ is difficult to determine for arbitrary semirings,  so we start with commutative semiring pairs of positive $e$-type.

\begin{lem}\label{prs1} $ $ In a  radical congruence $\Cong,$ \eroman
    \begin{enumerate}
        \item  $(b_1,b_2)\mtw (b_2,b_1) \in \Cong $ implies $(b_1,b_2) \in \Cong .$
        \item $(\one,b) \in \Cong$    if and only if $(\one + b^2, b+b) \in \Cong.$
    \end{enumerate}
\end{lem}
  \begin{proof}
    (i)  $(b_1,b_2)^\mtwt   = (-)(b_1,b_2)\mtw (b_2,b_1) \in \Cong $. Hence $(b_1,b_2) \in \Cong $.

      (ii)   $ (\one,b)\mtw(\one,b) = ( \one^2+b^2 , b+b) = ( \one+b^2 , b+b),$  so apply (i).
  \end{proof}


\begin{lem}\label{prs2} (Inspired by \cite[Proposition~2.10]{JM}): Suppose that $\Cong$ is  a congruence of an $e$-distributive  pair~$(\mcA,\mcA_0)$.
  \begin{enumerate}\eroman
      \item When $(\mcA,\mcA_0)$ is reduced, $  (1, e) \in \Cong $  if and only if $(\one + e + e , e+e)\in \Cong.$

\item $(\one,e), (e,\one) \in \sqrt{\Cong}$ if  $(\mcA,\mcA_0)/\Cong$ has  positive $e$-type.

       \item  $(\one,e), (e,\one) \in \sqrt{\Diag}$
      if   $(\mcA,\mcA_0)$ has positive $e$-type.
  \end{enumerate}
\end{lem}
\begin{proof}
    (i) Take $b = e$ in \lemref{prs1}(ii), noting  Lemma~\ref{esq}.

    (ii) We apply induction to (i).
    Namely, note for any $k'$, and $k'' = 2k'^2 +2k'$, that $$(\one + k'e, k'e)^\mtwt = ((\one + k'e)^2 +
(k'e)^2, 2k'e  +2k'e (\one _ k'e)     = (\one + k''e,k''e).$$ Suppose that  $(\mcA,\mcA_0)$ has   $e$-type  $k>0.$ Taking a high enough twist power of $(\one, e)$ gives $k'' >k,$ so $ (\one + k''e,k''e) = ( k''e,k''e) \in \Diag$ for all suitably large $k''$, and working backwards yields $(\one, e)\in \sqrt{\Diag}.$

    (iii) Special case of (ii).
\end{proof}

    \begin{thm}\label{rd} Suppose $(\mcA,\mcA_0)$ is  a   pair of positive $e$-type.
        \begin{enumerate}\eroman
            \item Every radical congruence of $\mcA$ contains $(\one, e)$.
            \item  When  $(\mcA,\mcA_0)$
 is $e$-central, $\hSpec (\mcA) $ is homeomorphic to $\hSpec (\mcA e).$
             \item Every maximal chain of   prime congruences of $\mcA[\la_1,\dots,\la_t]$ has length t.
        \end{enumerate}
    \end{thm}

\begin{proof}
    (i) By Lemma~\ref{prs2}.

    (ii) By (i) and  Lemma~\ref{tr1}.

    (iii) By (ii) and \cite[Theorem~4.6]{JM}.
\end{proof}

\begin{INote}
    \label{transf}
    \lemref{prs2}(iv) and Theorem~\ref{rd} provide a machinery to lift theorems from~\cite{JM}.
\end{INote}


\subsection{The general situation}\label{prop1}$ $

\begin{definition}
    A congruence $\Cong$ has $e$-\textbf{type} $k$ if $(\one+k e ,k e)\in \Cong$.

    $\Spece (\mcA,\mcA_0)$ is the space of prime congruences of positive $e$-{type}.
\end{definition}

\begin{thm}\label{sp2}
  For any $e$-central pair $(\mcA,\mcA_0)$ satisfying Property N,
    \begin{enumerate}
        \item  $\Spece (\mcA,\mcA_0)$ is isomorphic to $\hSpec ((\mcA,\mcA_0)/\Diag_e)$.

        \item If $\Cong$ is maximal with respect to not having positive $e$-type, then  $\Cong$ is a prime congruence.
    \end{enumerate}
\end{thm}
\begin{proof}
    (i) If $\Cong \in \Spece (\mcA,\mcA_0)$ then the pair $(\mcA,\mcA_0)/\Diag_e$ contains some element $(\one +ke,ke)$, and thus has   $e$-type $k$, so apply Theorem~\ref{rd}.

    (ii) If $\Cong_i \supset \Cong$ then by hypothesis $(\one +k_ie,k_ie)\Cong _i$ for suitable $i.$ But then $$(\one +k_1e,k_1e)\mtw(\one +k_2e,k_2e)  = ( \one + (2k_1k_2 +k_1+k_2)e, (2k_1k_2 +k_1+k_2)e,$$ implying $\Cong$ is of $e$-type $2k_1k_2 +k_1+k_2,$ a contradiction.
\end{proof}

\subsection{Proper congruences on pairs}$ $

 In Theorem~\ref{sp2} we get $\Spece$ from the Joo-Mincheva  Spectrum via Theorem~\ref{rd}, but in some sense this part is known, since the image pairs are degenerate. If $(\mcA,\mcA_0)$ does not have positive $e$-type, we are left with other  congruences,  some of which are prime.
   In this subsection we briefly consider such congruences, not necessarily prime.
In view of Lemma~\ref{id1} we exclude $(\one,e)$-congruences.
 \begin{definition}$ $
   \begin{enumerate}\eroman
 \item   An element $(a,b)$ of $\tT \times \mcA_0 $ is called \textbf{improper}. This  element $(a,b)$ is $\textbf{very improper}$ if $a +b = a.$

  \item A  congruence  $\Cong$  is \textbf{proper} if it does not contain any improper elements, i.e, if $(\mcA,\mcA_0)/\Cong$ is a proper pair.

   \item A  congruence  $\Cong$  is \textbf{weakly proper} if it does not contain any very improper elements.

  \item A  congruence  $\Cong$  is \textbf{prime proper} if it does not contain the product of two congruences each containing a very improper element.

     \end{enumerate}
 \end{definition}

\begin{example} $ $\begin{enumerate}
 \item The element $(\one  , ke)$ is very improper in a pair of $e$-type $k$.
    \item  Suppose that $(\mathbf{m},me)\in \Cong$ and $\mcA$ has $\mcA_0$-characteristic $k.$ Then $ \mcA/\Cong$  has $\mcA_0$-characteristic dividing $\gcd (m,k).$
\end{enumerate}
 \end{example}

\begin{lem}\label{pro3}
 Over an $e$-central pair, any  congruence   $ \Cong$ containing   $ (e,e^2) $ and  some element $(a, be)$ also contains $(a,ae).$
\end{lem}
\begin{proof}
    Suppose $(a, be)\in \Cong$. Then $ (ae,be^2) =(a,be) e  \in \Cong $.  Hence  $ (ae,be)   \in \Cong $  since $b(e,e^2) \in\Cong$,
  so by transitivity $(a,ae)\in \Cong.$
\end{proof}

\begin{corollary}
  Any $\tT$-cancellative  congruence of  an $e$-idempotent $e$-central pair containing  an improper element is a $(\one,e)$-congruence.
\end{corollary}

Hence, from now on we shall focus on proper and weakly proper congruences.
Notation is as in~Remark~\ref{congp}.

\begin{lem}\label{cp} If the pair $(\mcA,\mcA_0)$   is proper and the  congruence  $\Cong$   is proper, then $(\mcA,\mcA_0)/\Cong$ also is proper.

\end{lem}
 \begin{proof}

  We prove the contrapositive. For $a\in \tT,$ $\bar a \in \bar \mcA_0 $ if and only if $(a,b)\in \Cong$ for some $b\in \mcA_0.$
\end{proof}


\begin{lem} Any proper  congruence  $\Cong$ of the first kind on a shallow semiring pair  $(\mcA,\mcA_0)$ satisfies the property that $(a_1,a_2)\in \Cong$ for $a_i\in \tT$ implies $a_1+a_2\in \mcA_0$.
 \end{lem}
\begin{proof}
 By assumption, $(a_1,a_2)^\mtwt = (a_1^2+a_2^2, a_1 a_2+a_1 a_2)\in \Cong$. But $ a_1 a_2+a_1 a_2\in \mcA_0.$ Since $ \Cong$ is proper, $a_1^2+a_2^2 \in \mcA_0.$ It follows that $(a_1+a_2)^2 \in \mcA_0,$
so $a_1+a_2\in \mcA_0.$ \end{proof}


\begin{lem}\label{chains}$ $
 \begin{enumerate}\eroman
 \item The intersection of a proper congruence with other congruences is proper.
     \item The union of a chain $\Cong_i$ of proper congruences is a proper congruence. Consequently  any proper congruence  $\Cong$ is contained in a maximal proper congruence.

  \item  If $(a_i, b_i e)$ are very improper elements in a congruence of a semiring pair for $i = 1,2,$ then $(a_1, b_1 e)\mtw (a_2, b_2 e)$ is a very improper element.

     \item    Every maximal weakly proper congruence $\Cong$ of a semiring pair is weakly prime proper.
 \end{enumerate}
\end{lem}
\begin{proof}
    (i) Obvious.

    (ii) Clearly   $\cup \Cong_i$ is a congruence. If $(a,b)\in \cup \Cong_i$ for some $a\in \tT$ and $b\in {\mcA}_0$ then  $(a,b)\in \Cong_i, $ for some $i,$ a contradiction. The last assertion is by Zorn's Lemma.

   (iii) $a_i + b_i = a_i,$ for $i=1,2,$ so $a_1 a_2 +b_1b_2 =
   (a_1+b_1) a_2 + b_1b_2 =  a_1 a_2 +b_1(a_2+b_2)= (a_1+b_1)a_2 = a_1 a_2 . $  Hence  $(a_1, b_1 e)\mtw (a_2, b_2 e) =(a_1a_2, (a_1 b_2 +a_2 b_1)e).$

   (iv)  Suppose    $\Cong_1,  \Cong_2 \supset \Cong$ with
    $\Cong_1  \Cong_2 \subseteq \Cong.$ By definition there are  very improper elements   $(a_i,b_i) \in \Cong_i, $ so $\Cong$ has a very improper element, a contradiction.
\end{proof}


\begin{rem} One major example is the residue hypersemiring, which requires further examination.

 If $P\in \Spec (\mcA)$ and $G$ is a multiplicative subgroup of $\tT$ disjoint from $P,$ then the structure of  $\bar P:= {\{\sum_i p_ig_i: p_i\in P, \,g_i\in G\}} $ is unclear. If
    $\bar P\in \Spec(\mcA/G)  $ then $P\in \Spec(\mcA),$ but the other direction need not hold in general. Is there a nice positive result in this situation?
\end{rem}


\begin{thebibliography}{1}


\bibitem{AGG}
M.~Akian,   S.~Gaubert, and A.~Guterman, \emph{Linear independence
over tropical semirings and beyond}, Tropical and  Idempotent
Mathematics, G.L. Litvinov and S.N. Sergeev (eds). Contemp. Math.
\textbf{495}  (2009), 1--38.





\bibitem{AGR2}  M.~Akian, S. Gaubert, and L.~Rowen,
\emph{Semiring systems arising from hyperrings}, Journal Pure Applied Algebra~\textbf{228} (2024), 107584.



 \bibitem{AGR1}  M.~Akian, S. Gaubert, and L.~Rowen,
\emph{Linear algebra over $\tT$-pairs} (2023),  arXiv 2310.05257.







\bibitem{BL}
 M. Baker and O. Lorscheid,
  \emph{Descartes' rule of signs, {N}ewton polygons, and polynomials
              over hyperfields},
    Journal of Algebra \textbf{569},  416--441 (2021).


\bibitem{FM} N.~Friedenberg and K.~Mincheva,  \emph{ Tropical adic spaces I: the continuous spectrum of a topological semiring}, Res Math Sci \textbf{11}, 55 (2024). https://doi.org/10.1007/s40687-024-00467-6

\bibitem{GaR} L.~Gatto, and L.~Rowen,
\emph{Grassman semialgebras and the Cayley-Hamilton theorem},
Proceedings of the American Mathematical Society, series B \textbf{7} (2020),
 183–-201.



\bibitem{golan92} J.~Golan.
\emph{The theory of semirings with applications in mathematics and
theoretical computer science}, Longman Sci \&\ Tech., volume~54,
(1992).



\bibitem{I} Z.~Izhakian,   \emph{Commutative $\nu$-algebra and supertropical algebraic geometry}, https://arxiv.org/pdf/1901.08032.


\bibitem{IR}
Z.~Izhakian, and L.~Rowen. \emph{Supertropical algebra}, { Advances
in Mathematics}, \textbf{225}(2010), 2222--2286.

%






\bibitem{JM} D.~Jo{\'o}, and K.~Mincheva, \emph{Prime congruences of additively idempotent semirings and a Nullstellensatz for tropical polynomials}, Selecta Mathematica \textbf{24}
 (3), (2018), 2207--2233.





 \bibitem{JMR} J.~Jun, K.~Mincheva, and L.~Rowen, \emph{$\mathcal T$-semiring pairs},  volume in honour of Prof. Martin Gavalec,
Kybernetika   \textbf{58}  (2022),   733--759.

\bibitem{JuR1}  J.~Jun, and L.~Rowen,
\emph{Categories with negation},  in Categorical, Homological and
Combinatorial Methods in Algebra
 (AMS Special Session in honor of S.K. Jain's 80th birthday), Contemporary
Mathematics \textbf{751} (2020), 221-270.




\bibitem{krasner}
M.~Krasner,
{\it A class of hyperrings and hyperfields},
Internat. J. Math. \& Math. Sci.
6 no. 2, (1983),
 307--312.



\bibitem{Lo} O.~Lorscheid {\it $F_1$ for everyone},
 arXiv:1801.05337v1.

\bibitem{Mar}
 F.Marty,
  {\it R\^ole de la notion de hypergroupe dans i’ tude de groupes non abeliens},
 Comptes Rendus Acad. Sci. Paris 201, (1935),
 {636--638}.

\bibitem{Mas} Massouros, C.G., Methods of constructing hyperfields. Int. J. Math. Math. Sci. 1985, 8, 725–728.

\bibitem{MasM}
 {C.~Massouros and G.~Massouros},
 {\it On the Borderline of
Fields and Hyperfields},
 Mathematics 11, (2023), {1289}, URL= {https://doi.org/
10.3390/math11061289}.


\bibitem{NakR} S.~Nakamura, and M.L.~Reyes,
{\it Categories of hypersemirings, hypergroups, and related hyperstructures} (2024), arXiv:2304.09273.


\bibitem{Row21} L.H.~Rowen,
\emph{Algebras with a negation map}, European J. Math.  Vol.~{\textbf 8} (2022), 62--138. https://doi.org/10.1007/s40879-021-00499-0,
arXiv:1602.00353.


\bibitem{Row24b}  L.H.~Rowen,  \emph{Tensor products of bimodules over monoids}, http://arxiv.org/abs/2410.00992
 (2024).


\bibitem{V} {O. Ya. Viro},
\emph{Hyperfields for tropical geometry I. Hyperfields and dequantization},
http://arxiv.org/abs/1006.3034 (2010).

 \end{thebibliography}
\end{document}